\documentclass[11pt]{amsart}
\usepackage{amsmath}
\usepackage{amsfonts}
\usepackage[mathscr]{eucal}
\usepackage{amssymb}
\usepackage{marginnote}
\usepackage{geometry,amsthm,amsmath,amssymb, graphicx,  float, enumerate}
\usepackage{color}
\usepackage[all]{xy}
\newtheorem{thm}{Theorem}[section]

\newtheorem{cor}[thm]{Corollary}

\newtheorem{lemma}[thm]{Lemma}
\newtheorem{prop}[thm]{Proposition}

\theoremstyle{definition}
\newtheorem{defn}[thm]{Definition}
\newtheorem{remark}[thm]{Remark}

\newcommand{\Span}{\operatorname{span}}
\newcommand{\bb}[1]{\mathbb{#1}}
\newcommand{\cl}[1]{\mathcal{#1}}
\newcommand{\ff}[1]{\mathfrak{#1}}

\newcommand{\Ran}{\operatorname{Ran}}

\newcommand{\Ker}{\operatorname{Ker}}

 \allowdisplaybreaks[3]
\begin{document}

\title[The Operator System Generated by Cuntz Isometries]{The Operator System Generated by Cuntz Isometries}


\author[Da Zheng]{Da Zheng}
\address{Department of Mathematics, University of Houston, 4800 Calhoun Road, PGH Hall, office 679, TX 77204}
\email{dazheng@math.uh.edu}
\subjclass[2000]{ Primary 46L06, 46L07; Secondary 46L05, 47L25.}
\thanks{This research was supported in part by NSF grant DMS-1101231.}
\date{\today}
\begin{abstract}
In this paper we consider the operator system $\cl{S}_n$ generated by $n$ Cuntz isometries, i.e. the span of the generators of the Cuntz algebra $\cl{O}_n$ 
together with their adjoints and the identity. 
We define an operator subsystem $\cl{E}_n\subseteq M_{n+1}$ and then prove that $\cl{S}_n$ is completely order
isomorphic to a quotient of $\cl{E}_n$. This result implies a characterization of positive elements in $M_p(\cl{S}_n)$.
\end{abstract}
\maketitle
 \section{Introduction}
 Farenick and Paulsen \cite{farenick2011operator} studied the operator system $\ff{S}_n$ generated by $n$ universal unitaries in $C^*(F_n)$, where 
 $F_n$ is the discrete free group on $n$ generators, and they showed that 
 $\ff{S}_n$ is completely order isomorphic to a quotient of the operator system $\ff{T}_n$ of tridiagonal matrices \cite[Theorem 4.2]{farenick2011operator}. By using this isomorphism,
 they gave a new equivalent condition for a $C^*$-algebra to have WEP (weak expectation property), which they called property $\ff{S}$ \cite[Theorem 6.2]{farenick2011operator}.  
 In that paper,they also
 defined an operator system $\ff{W}_n\subseteq C^*(F_n)$ and proved that  $\ff{W}_n$ is completely order isomorphic to a quotient of the matrix algebra
 $M_n$ and this leads to a new proof of Kirchberg's theorem  \cite[Corollary 3.2]{farenick2011operator}. 
 
 The Cuntz algebra is generated by universal isometries with certain relations. So one may hope that the universal property of these isometries can give us some analogous  results 
 as mentioned in the last paragraph. This motivates us to study the operator system generated by Cuntz isometries, which will be 
 denoted by $\cl{S}_n$ throughout, and utilize the universal property of the isometries  to derive properties of $\cl{S}_n$. 
 Moreover, we may hope that $\cl{S}_n$ can offer us new characterizations of 
 some important classes of $C^*$-algebras, even new classes of $C^*$-algebras. 
 
 Thus, we start exploring $\cl{S}_n$ in this paper. In section 3, we focus on $\cl{S}_n$ itself and prove that it has  a universal property by observing that $\cl{S}_n$ 
 is completely order isomorphic to $\hat{\cl{S}}_n$--the operator system generated by the Teoplitz-Cuntz isometries. In section 4, we define an
 operator subsystem $\cl{E}_n$ in the matrix algebra $M_{n+1}$ and construct a unital completely positive map from $\cl{E}_n$
 to $\cl{S}_n$. In section 5, we prove that $\cl{S}_n$ is completely order isomorphic to a quotient of $\cl{E}_n$ and 
give a characterization of positive elements in $M_p(\cl{S}_n)$.
 
  \begin{remark}\label{introremark}
  By the uniqueness of the Cuntz algebra, we know that if $\cl{S}_n$ and $\cl{T}_n$ are two operator systems generated by $n$ Cuntz isometries, then 
  $\cl{S}_n$ is unitally completely order isomorphic to $\cl{T}_n$. Conversely, 
if $\cl{S}_n$ is unitally completely order isomorphic to $\cl{T}_n$, then $C^*(\cl{S}_n)\cong C^*(\cl{T}_n)$, which would lead to a new proof of the uniqueness of the Cuntz algebra. 
This reveals that we may prove some properties of the Cuntz algebra via
 $\cl{S}_n$. Conversely, it is also interesting to explore how many nice properties $\cl{S}_n$ can inherit from $\cl{O}_n$.
  \end{remark}

 
 \section{Preliminaries}
 \subsection{Operator Systems and Completely Positive Maps}
 These concepts are quite familiar to most people so we mention the very basic ones and the notations which we will adopt throughout the paper for the sake of convenience.
 See \cite{paulsen2002completely} for further details of this topic. 
  \begin{defn}[Concrete Operator System]
 A \textbf{concrete operator system} $\cl{S}$ is a unital $*$-closed  subspace of some unital $C^*$-algebra $\cl{A}$, that is,
 $\cl{S}\subseteq \cl{A}$ is a subspace of $\cl{A}$ such that $a\in \cl{S}\Rightarrow a^*\in \cl{S}$ and $1_{\cl{A}}\in \cl{S}$. Here, $1_{\cl{A}}$ denotes
 the unit of $\cl{A}$. 
 \end{defn}
Henceforth, we shall denote the unit of an arbitrary operator system $\cl{S}$ by $1_{\cl{S}}$. 
  \begin{defn}[Abstract Operator System]
 An \textbf{abstract operator system} $\cl{S}$ is a matrix-ordered $*$-vector space with an Archimedean matrix order unit.
 \end{defn}
  We write $M_n(\cl{S})^+$, $n\in \bb{N}$ for the positive cones of $\cl{S}$ and $(a_{ij})\geq 0$ if  $(a_{ij})\in M_n(\cl{S})^+$.
    \begin{defn}
Let $\cl{S}$ and $\cl{T}$ be operator systems. A linear map $\phi:\cl{S}\to \cl{T}$ is called  \textbf{completely positive} if
\[ \phi^{(n)}((a_{ij})):=(\phi(a_{ij}))\geq 0, \quad \text{ for each } (a_{ij})\in M_n(\cl{S})^+ \text{ and for all } n\in \cl{N} .\]
 \end{defn}
  \begin{defn}
Let $\cl{S}$ and $\cl{T}$ be operators systems. A map $\phi:\cl{S}\to \cl{T}$ is called a \textbf{complete order isomorphism} if $\phi$ is a unital linear 
isomorphism and both $\phi$ and $\phi^{-1}$ are  completely positive, and we say that $\cl{S}$ is \textbf{completely order isomorphic} to $\cl{T}$ if such $\phi$
exists.
\end{defn}
\begin{remark}
 Indeed, the above definition of a complete order isomorphism is for ``unital complete order isomorphisms''. Since we will only use unital complete order 
 isomorphisms in this paper, it is convenient to omit ``unital'' in the above definition. 
\end{remark}

\begin{defn}
 A map $\phi$ is called a \textbf{complete order injection} if it is a complete order isomorphism onto its range with $\phi(1_{\cl{S}})$ being an
Archimedean order unit. We shall denote this by $\cl{S}\subseteq_{\operatorname{c.o.i}}\cl{T}$. 
\end{defn}
The following theorem indicates that every abstract operator system is completely order isomorphic to a concrete one. So from now on, we will not distinguish a 
abstract operator system from a concrete one and will call them just operator systems.
  \begin{thm}[Choi, Effros]\label{choieffros}
  Let $\cl{S}$ be an abstract operator system, then there exists a Hilbert space $\cl{H}$, a concrete operator system $\cl{S}_1\subseteq B(\cl{H})$,
  and a unital complete order isomorphism $\varphi:\cl{S}\to \cl{S}_1$. Conversely, a concrete operator system is also an abstract operator system.
 \end{thm}
  \subsection{Operator System Quotients and Complete Quotient Maps}
 In this section, we review the definition of the operator system quotients and complete quotient maps as well as some basic facts which shall be used in our later discussion.
 Most of these can be found in \cite{kavruk2013quotients}. 

 \begin{defn}
 Given an operator system $\cl{S}$, we call $J\subseteq \cl{S}$ a \textbf{kernel}, if $J=\ker\phi$ for an operator system $\cl{T}$ and
 some unital completely positive map $\phi:\cl{S}\to \cl{T}$.
\end{defn}
\begin{prop}\label{quotientoperatorsystemprop}
Let $\cl{S}$ be an operator system and $J\subseteq \cl{S}$  be  kernel, if  we define a family of matrix cones on $\cl{S}/J$ by setting
 \begin{align*}
  C_n=&\{(x_{ij}+J)\in M_n(\cl{S}/J): \text{ for each } \epsilon>0, \text{ there exists } (k_{ij})\in M_n(J) \\
  &\text{ such that } \epsilon \otimes I_n+(x_{ij}+k_{ij})\in  M_n(\cl{S})^+\}.
 \end{align*}
 then $(\cl{S}/J,\{C_n\}_{n=1}^{\infty})$ is a matrix ordered $*$-vector space with an Archimedean matrix unit $1+J$, and the quotient map $q:\cl{S}\to \cl{S}/J$ is completely positive.
 \end{prop}'
 This shows that $(\cl{S}/J,\{C_n\}_{n=1}^{\infty}),1+J)$ is an abstract operator system and is therefore an concrete one by Theorem \ref{choieffros}.
\begin{defn}
Let $\cl{S}$ be an operator system and $J\subseteq \cl{S}$  be  kernel. We call the operator system $(\cl{S}/J,\{C_n\}_{n=1}^{\infty}),1+J)$
defined in Proposition \ref{quotientoperatorsystemprop} the \textbf{quotient operator system}.
\end{defn}

\begin{defn}
 Let $J$ be a kernel and define
 \[ D_n=\{(x_{ij}+J)\in M_n(\cl{S}/J): \text{ there exists } y_{ij}\in J \text{ such that } (x_{ij}+y_{ij})\in M_n(\cl{S})^+\}, \]
 then $J$ is \textbf{completely order proximinal} if $C_n=D_n$ for all $n\in \bb{N}$.
\end{defn}
\begin{defn}
 Let $\cl{S}$, $\cl{T}$  be operator systems and $\phi:\cl{S}\to \cl{T}$ be  a completely positive map, then $\phi$ is called a \textbf{complete quotient map} if
 $\cl{S}/\Ker\phi$ is completely order isomorphic to $\cl{T}$.
\end{defn}
\subsection{Cuntz Algebra and Toeplitz-Cuntz Algebra}
In this section, we review some fundamental facts about the Cuntz algebra as well as the Toeplitz-Cuntz algebra. See \cite{cuntz1977simplec} for more details.

The Cuntz algebra $\cl{O}_n$ ($2\leq n<+\infty$) is the universal $C^*$-algebra generated by $n$ isometries $S_1,\dots,S_n$ with $\sum_{i=1}^nS_iS^*_i= I$, where 
$I$ is the identity operator. The algebra $\cl{O}_n$ enjoys the following universal property:
If $T_1,\cdots,T_n$ are $n$ isometries satisfying $\sum_{i=1}^nT_iT^*_i= I$, then there is a surjective $*$-homomorphism
$\pi:C^*(S_1,\dots,S_n)\to C^*(T_1,\dots,T_n)$ such that $\pi(S_i)=T_i$ for $1\leq i\leq n$.

The Cuntz algebra $\cl{O}_\infty$ is the universal $C^*$-algebra generated by the isometries $\{S_i\}_{i=1}^\infty$ with $\sum_{i=1}^nS_iS^*_i< I$,  for every $n\in \bb{N}$. 
It has similar universal properties.

The Toeplitz-Cuntz algebra $\cl{TO}_n$ ($2\leq n<+\infty$) is the universal $C^*$-algebra generated by $n$ isometries $S_1,\cdots,S_n$ with
$\sum_{i=1}^nS_iS^*_i\leq I$, and it has the following universal property:
If $T_1,\cdots,T_n$ are $n$ isometries satisfying $\sum_{i=1}^nT_iT^*_i\leq I$, then there is a surjective $*$-homomorphism
$\pi:C^*(S_1,\dots,S_n)\to C^*(T_1,\dots,T_n)$ such that $\pi(S_i)=T_i$ for $1\leq i\leq n$.

\begin{remark}
We note that any $n$ ($1\leq n<+\infty$) isometries $S_1,\cdots,S_n$ with
$\sum_{i=1}^nS_iS^*_i< I$ can be  the generators of $\cl{TO}_n$. This is due to the famous Cuntz-Krieger uniqueness theorem. 
\end{remark}
\section{The Operator System  $\cl{S}_n$ Generated by Cuntz Isometries}
 For $2\leq n<+\infty$, let $S_1,\dots,S_n$ be the isometries that generate $\cl{O}_n$, $I$ be the identity,  and we define the operator system $\cl{S}_n$ as:
 \[ \cl{S}_n:=\Span\{I, S_i,S_i^*:1\leq i\leq n\} .  \]
 Similarly, we denote $\cl{S}_\infty$ as the operator system generated by the isometries that generate $\cl{O}_\infty$. Also, for $n=1$, we let $\cl{S}_1$ be the three-dimensional
 operator system generated by the a universal unitary (for example, $z\in C(\bb{T})$). 
 
 Also, let $\hat{S}_1,\cdots,\hat{S}_n$ be $n$ ($1\leq n<+\infty$) isometries with $\sum_{i=1}^n\hat{S}_i\hat{S}^*_i<I$ and set
\[ \hat{\cl{S}}_n=\Span\{ I, \hat{S}_1,\dots,\hat{S}_n, \hat{S}^*_1,\dots,\hat{S}^*_n\} , \]
so that $\hat{\cl{S}}_n$ is the operator system generated by the Toeplitz-Cuntz isometries. 
 \begin{remark}
 It is mentioned in Remark \ref{introremark} that any $n$ isometries satisfying the Cuntz relation give rise to the same $\cl{S}_n$. Similarly, the universal property of 
Toeplitz-Cuntz algebra also implies that
any $n$ isometries satisfying the Toeplitz-Cuntz relation give rise to the same $\hat{\cl{S}}_n$.
 \end{remark}

 \begin{lemma}\label{ToeplitztoCuntz}
 Suppose $T_1,\dots,T_n$ are $n$ ($2\leq n\leq +\infty$) isometries  on a Hilbert space $\cl{H}$ with  $\sum_{i=1}^nT_iT^*_i<I_{\cl{H}}$, then they can be dilated to
 $n$ isometries $\tilde{T}_1,\dots, \tilde{T}_n$ on some Hilbert space $\cl{K}$ with $\sum_{i=1}^n\tilde{T}_i\tilde{T}^*_i=I_{\cl{K}}$.
 \end{lemma}
\begin{proof}
Let $M=\Ran (\sum_{i=1}^nT_iT^*_i)$ and hence $M^\perp=\Ran (I_{\cl{H}}-\sum_{i=1}^nT_iT^*_i)$. Note that since $T_i$'s are isometries, $\dim M^\perp\leq \dim M$. 
 Let $P_{M\perp}$ be the projection onto $M^\perp$ and  we define operators $X_i: \cl{H}\to \cl{H}$ as the following:
 \[ X_i=\begin{cases} 
      P_{M\perp} & i=1 \\
      0 & 2\leq i\leq n
   \end{cases}.\]
   Correspondingly, we choose $n$ operators $Y_i:\cl{H}\to \cl{H}$ where $Y_1$ is a partial isometry with initial space $M$ and $Y_i$ ($2\leq i\leq n$) are isometries
such that $\sum_{i=1}^nY_iY_i^*=I_{\cl{H}}$. 

Next, let $\cl{K}=\cl{H}\oplus \cl{H}$, we define $\tilde{T}_i:\cl{K}\to \cl{K}$ by 
\[ \tilde{T}_i=\begin{pmatrix}
                T_i&X_i\\
                0&Y_i
               \end{pmatrix}, \]
               and it is easy to check that 
               \begin{align*}
           \tilde{T}_i^*\tilde{T}_i&=\begin{pmatrix}
                                      T_i^*T_i&T_i^*X_i\\
                                      X_i^*T_i&X_i^*X_i+Y_i^*Y_i
                                     \end{pmatrix}=\begin{pmatrix}
                                     I_{\cl{H}}&0\\
                                     0&I_{\cl{H}}
                                     \end{pmatrix}  \\
    \sum_{i=1}^n\tilde{T}_i\tilde{T}_i^*&=\sum_{i=1}^n\begin{pmatrix}
                                                       T_iT_i^*+X_iX_i^*&X_iY_i^*\\
                                                       Y_iX_i^*&Y_iY_i^*
                                                      \end{pmatrix}=\begin{pmatrix}
                                                       I_{\cl{H}}&0\\
                                     0&I_{\cl{H}}
                                     \end{pmatrix}.
\end{align*}
Hence, $\tilde{T}_1,\dots, \tilde{T}_n$ are the desired dilations.
\end{proof}
\begin{cor}\label{CuntzandToeplitzCuntzidentical}
 The operator system $\hat{\cl{S}}_n$ is unitally completely isomorphic to $\cl{S}_n$ in the canonical way, that is, there exists a unital complete order isomorphism 
 $\phi:\hat{\cl{S}}_n\to \cl{S}_n$ such that $\phi(\hat{S}_i)=S_i$. 
 \end{cor}
\begin{proof}
Since it is easy to see that Cuntz isometries can also be dilated to Toeplitz-Cuntz isometries, the corollary follows by the fact that compressions are completely positive.
\end{proof}
\begin{cor}
We have that  if $1\leq n<m\leq +\infty$, then $\cl{S}_n\subseteq_{\operatorname{c.o.i}}\cl{S}_m$ via the natural embedding.  
\end{cor}
\begin{proof}
 This is immediate from the last corollary.
\end{proof}
\begin{cor}
For each $1\leq i\leq n$ $1\leq n\leq +\infty$, we define $\rho: \cl{S}_n\to \cl{S}_n$, $X\mapsto S_i^*XS_i$, then $\rho$ is a complete positive projection whose range is completely 
order isomorphic to $\cl{S}_1$.
\end{cor}
\begin{proof}
 Clearly $\rho$ is completely positive. Due to the Cuntz relation, it is easily seen that $\rho\circ \rho=\rho$, and $\Ran \rho=\Span\{I, S_i,S_i^*\}$. However,
$S_i$ is a single Toeplitz-Cuntz isometry and hence $\Span\{I, S_i,S_i^*\}=\cl{S}_1$ completely order isomorphically.
\end{proof}

\begin{defn}
 The $n$-tuple of  operators $(A_1,\dots,A_n)$ is called a \textbf{row contraction} if $\sum_{i=1}^nA_iA_i^*\leq I$ ,where $I$ is the identity operator.  
\end{defn}
Bunce proved in \cite{bunce1984models} that any family of $n$ operators $\{A_i\}_{i=1}^n$ with $\sum_{i=1}^nA_i^*A_i\leq I$  can be dilated to $n$ coisometries 
with orthogonal initial spaces. Here, we rephrase that proposition for isometric dilation of row contractions. 
\begin{prop}\label{Buncedilation}
 Let $(A_1,\dots,A_n)$ be a row contraction on some Hilbert space $\cl{H}$, then there exists isometries $W_1,\dots,W_n$ with 
 $W_i^*W_j=0$ if $i\neq j$, such that $P_{\cl{H}}W_i|_{\cl{H}}=A_i$. 
\end{prop}

 The above proposition implies the following universal property of $\cl{S}_n$:
 \begin{thm}\label{universalproperty}
 The operator system $\cl{S}_n$ has the following universal property:
 \begin{center}
  \noindent Let $(A_1,\cdots,A_n)$ be a row contraction on some Hilbert space $\cl{H}$ and denote 
  \[\cl{T}_n=\Span\{I_\cl{H},A_i,A_i^*:1\leq i\leq n\}\]
  so that $\cl{T}_n$ is an operator system,  
  then there exists a unital completely positive map $\phi:\cl{S}_n\to \cl{T}_n$ such that $\phi(S_i)=A_i$. 
 \end{center}
\end{thm}
\begin{proof}
 By Proposition \ref{Buncedilation}, we can dilate $A_1,\cdots,A_n$ to $W_1,\dots,W_n$ with orthogonal ranges, i.e. $\sum_{i=1}^nW_iW_i^*\leq I$. Let 
 $W_n=\Span\{I,W_i,W_i^*:1\leq i\leq n\}$, then the 
 universal property of the Teoplitz Cuntz algebra implies that there exists a unital completely positive map from $\hat{\cl{S}}_n$ to $W_n$ which sends 
 $\hat{S}_i$ to $W_i$. As compressions are always completely positive, we have a unital completely positive map from $\hat{\cl{S}}_n$ to $\cl{T}_n$ mapping
 $\hat{S}_i$ to $A_i$.
 Now the conclusion follows from Remark \ref{CuntzandToeplitzCuntzidentical}. 
\end{proof}
 \section{The operator system $\cl{E}_n$ and Its Quotinet}
We define an operator system $\cl{E}_n\subseteq M_{n+1}:=M_{n+1}(\bb{C})$ as the following,
\[ \cl{E}_n=\Span\{E_{00},E_{0i},E_{i0},\sum_{i=1}^nE_{ii}: 1\leq i\leq n \}, \]
where $E_{ij}$'s are matrix units in $M_{n+1}$. So every element in $\cl{E}_n$ is of the form,
\[ \begin{pmatrix}
    a_{00}&a_{01}&\cdots&a_{0,n}\\
    a_{10}&b&&\\
    \vdots&&\ddots&\\
    a_{n0}&&&b
   \end{pmatrix} .\]
    By representing the Cuntz isometries $S_1,\dots,S_n$ on some Hilbert space $\cl{H}$, we define an operator $R: \cl{H}^{(n+1)}\to \cl{H}$ by
    $R:=(\frac{\sqrt{2}}{2}I,\frac{\sqrt{2}}{2}S_1^*,\dots,\frac{\sqrt{2}}{2}S_n^*)$. So we know that
   \[ R^*R=\begin{pmatrix}
           \frac{1}{2} I&\frac{1}{2}S_1^*&\cdots&\frac{1}{2}S_n^*\\
            \frac{1}{2}S_1&&&\\
            \vdots&&\bigl(\frac{1}{2}S_iS_j^*\bigr)&\\
            \frac{1}{2}S_n&&&
           \end{pmatrix},  \]
 is positive in $M_{n+1}(B(\cl{H}))$ .
 
Now, we can define a map $\psi:M_{n+1}\to B(\cl{H})$ by
\[ \psi(E_{ij})=(R^*R)_{ij}, \]
where $(R^*R)_{ij}$ denotes the $i,j$-th entry of $R^*R$,
and extend it linearly to $M_{n+1}$. It is straight forward that $\phi$ is unital.
\begin{thm}[Choi]
  Let $\cl{A}$ be a $C^*$-algebra, $\phi:M_n\to \cl{A}$ be linear, and $\{E_{ij}\}$ be the standard matrix units for $M_n$, then the following are equivalent:
  \begin{enumerate}
   \item $\phi$ is completely positive.
   \item $\phi$ is $n$-positive.
   \item $\bigl(\phi(E_{ij})\bigr)_{i,j=0}^n$ is positive in $M_n(\cl{A})$.
  \end{enumerate}
\end{thm}
Since
\[ \bigl(\psi(E_{ij})\bigr)_{i,j=0}^n=R^*R\geq 0, \] 
 Choi's theorem tells us that $\psi$ is unitally completely positive.

Next, we can easily calculate that $\Ker \psi=\Span\{E_{00}-\sum_{i=1}^nE_{ii}\}$, which will be denoted as $J$ throughout the rest of the paper.
 \begin{prop}\cite{kavruk2011nuclearity}\label{finitedimensionalproximinal}
  Let $J$ be a finite dimensional $*$-subspace in a operator system $\cl{S}$ which contains no positive elements other than $0$, then it is
  a completely order proximinal kernel.
 \end{prop}

\begin{lemma}\label{Jcompleteorderproximimal}
 The kernel  $J$ it  is completely order proximinal.
\end{lemma}
\begin{proof}
According to Proposition \ref{finitedimensionalproximinal}, we just need to show that $J$ contains no positive elements other than $0$. This is clear
by considering the first and second diagonal entries of any nonzero element in $J$. Hence, $J$ is completely order proximinal.
\end{proof}

Now, let $\phi=\psi|_{{\cl{E}}_n}$, then $\phi:\cl{E}_n\to \cl{S}_n$ is unitally completely positive whose kernel is also $J$. By Proposition \ref{quotientoperatorsystemprop}
we can form the operator system quotient $\cl{E}_n/J$, and the induced map 
\[ \tilde{\phi}:\cl{E}_n/J\to \cl{S}_n, \quad  x+J\mapsto \phi(x) \]
is unitally completely positive and bijective.  
\section{$\cl{E}_n/J=\cl{S}_n$ and a characterization of positive elements in $M_p(\cl{S}_n)$}
We now state the main result of this paper:
\begin{thm}\label{ScompleteorderisomorphictoEJ}
 We have that $\cl{E}_n/J$ is completely order isomorphic to $\cl{S}_n$ and hence $\phi$ is a complete quotient map.
\end{thm}

  To prove the theorem, we need to show that $\tilde{\phi}^{-1}$ is also completely positive, which will imply  that $\tilde{\phi}$ is a complete 
order isomorphism.

To this end, we first notice the following from the definition $ \tilde{\psi}$:
\begin{align*}
 \tilde{\phi}^{-1}(S_i)&=2E_{i0}+J, \\
   \tilde{\phi}^{-1}(I)&=I_{n+1}+J=2E_{00}+J=2\sum_{i=1}^nE_{ii}+J.
\end{align*}
By Theorem \ref{choieffros}, we can always embed $\cl{E}_n/J$ in $B(\cl{K})$ completely order isomorphically. Let the embedding be $\gamma$, and denote
\[ T_i:=\gamma(2E_{i0}+J),\quad  I_\cl{K}:=\gamma(I_{n+1}+J),\]
it is equivalent to show that the map 
\[\hat{\phi}:\cl{S}_n\to B(\cl{K}), \quad S_i\mapsto T_i, \quad S_i^*\mapsto T_i^*,\quad I\mapsto I_\cl{K} \]
is completely positive.

 The proof of the following lemma is quite similar to that of Lemma 3.1 in \cite{paulsen2002completely}, so we omit the proof.
\begin{lemma}
 Let $\cl{H},\cl{K}$ be Hilbert spaces and $T\in B(\cl{H},\cl{K})$. Also, denote $I_{\cl{H}}$ and $I_{\cl{K}}$ as the Identity operators on $\cl{H}$ and $\cl{K}$ respectively. Then
 we have $\|T\|\leq 1$ if and only if
 \[\begin{pmatrix}
    I_{\cl{H}}&T^*\\
    T&I_{\cl{K}}
   \end{pmatrix} \]
is positive in $B(\cl{H}\oplus \cl{K})$.
\end{lemma}
Using this lemma, we can prove the following.
\begin{prop}\label{Tnrowcontraction}
  We have that $(T_1,\dots,T_n)$ is a row contraction, i.e., $\sum_{i=1}^nT_iT_i^*\leq I_\cl{K}$.
 \end{prop}
\begin{proof}
We represent $T_i$'s on a Hilbert space $\cl{H}$ via some unital one-to-one $*$-homomorphism, so $(T_1,\dots,T_n)\in B(\cl{H}^{(n)},\cl{H})$.
By the above lemma, equivalently, we show the following,
\[ \begin{pmatrix}
    I_\cl{K}&T_1&\cdots&T_n \\
    T_1^*&I_\cl{K}&&  \\
    \vdots&&\ddots&  \\
    T_n^*&&&I_\cl{K}
   \end{pmatrix}\geq 0.
\]
Notice that  $E_{00}+\sum_{i=1}^nE_{ii}=I_{n+1}$ and $E_{00}+J=\sum_{i=1}^nE_{ii}+J$, so
\begin{equation*}
 E_{00}+J=\sum_{i=1}^nE_{ii}+J=\frac{1}{2}I_{n+1}+J
\end{equation*}.

Since the quotient map $q:M_{n+1}\to M_{n+1}/J$ is completely positive,
we just need to show
\[ \begin{pmatrix}
    2\sum_{i=1}^nE_{ii}&2E_{10}&\cdots&2E_{n0}  \\
    2E_{01}&2E_{00}&&\\
    \vdots&&\ddots& \\
   2 E_{0n}&&&2E_{00}
   \end{pmatrix}\geq 0  .\]
 To this end, we write this matrix as a sum of $n$ matrices
 \begin{align*}
  \begin{pmatrix}
    \sum_{i=1}^nE_{ii}&E_{10}&\cdots&E_{n0}  \\
    E_{01}&E_{00}&&\\
    \vdots&&\ddots& \\
    E_{0n}&&&E_{00}
   \end{pmatrix}=&\begin{pmatrix}
  E_{nn}&0&\cdots&0&E_{n0}  \\
    0&\ddots&&&0\\
    \vdots&&\ddots&&\vdots \\
    0&&&\ddots&0 \\
    E_{0n}&0&\cdots&0&E_{00}
   \end{pmatrix}\\
   &+\begin{pmatrix}
  E_{n-1,n-1}&0&\cdots&0&E_{n-1,0}&0  \\
   0&0&\cdots&\cdots&0&0\\
    \vdots&&\ddots&&\vdots &\vdots\\
     0&&&0&0&\vdots\\
      E_{0,n-1}&0&\cdots&0&E_{00}&0 \\
    0&0&\cdots&\cdots&0&0
   \end{pmatrix}+\\
   &\cdots+\begin{pmatrix}
  E_{11}&E_{10}&0&\cdots&0  \\
    E_{01}&E_{00}&0&\cdots&0\\
    0&0&0&\cdots&0 \\
    \vdots&\vdots&&\ddots&\vdots \\
    0&0&\cdots&\cdots&0
   \end{pmatrix}.
 \end{align*}
From the equation above, we can see that  each summand on the right is positive, so the block matrix on the left is positive and the conclusion follows.
 \end{proof}
Now Theorem \ref{ScompleteorderisomorphictoEJ} is straightforward.
 \begin{proof}[Proof of Theorem \ref{ScompleteorderisomorphictoEJ}]
  By Proposition \ref{Tnrowcontraction} and Theorem \ref{universalproperty}, we know that there exists unital completely positive map which sends 
  $S_i$ to $T_i$. However, this maps is necessarily $\hat{\phi}$. 
 \end{proof}
Theorem \ref{ScompleteorderisomorphictoEJ} implies the following corollary which gives a characterization for positive elements in $M_p(\cl{S}_n)$.

\begin{cor}
  We have that $A_0\otimes I+\sum_{i=1}^nA_i\otimes S_i+\sum_{i=1}^nA_i^*\otimes S_i^*\in M_p(\cl{S}_n)^+$  if and only if there exists $B\in M_p$ such that
 \[ \begin{pmatrix}
     A_0&2A_1^*&\cdots&2A_n^* \\
     2A_1&A_0&& \\
     \vdots&&\ddots&\\
     2A_n&&&A_0
    \end{pmatrix}+
    \begin{pmatrix}
     B&&&\\
     &-B&&\\
     &&\ddots&\\
     &&&-B
    \end{pmatrix}\in M_{n+1}(M_p)^+
\]
\end{cor}
\begin{proof}
 By Theorem \ref{ScompleteorderisomorphictoEJ}, equivalently, we consider 
 \[M=A_0\otimes (I_n+J)+\sum_{i=1}^nA_i\otimes (2E_{i0}+J)+\sum_{i=1}^nA_i^*\otimes (2E_{0i}+J)\in M_p(\cl{E}_n/J)^+ .\]
 
If $A_k=(a_{ij}^k)_{i,j=1}^p$ for $0\leq k\leq n$, then
\[ M=\biggl(a_{ij}^0(I_{n+1}+J)+\sum_{k=1}^na_{ij}^k(2E_{k0}+J)+\sum_{k=1}^n\overline{a^k_{ji}}(2E_{0k}+J)\biggr)_{i,j=1}^p. \]

Lemma \ref{Jcompleteorderproximimal} together with definition of positivity in a quotient operator system imply that
there exists $(J_{ij})\in M_p(J)$ such that
\[M= (a_{ij}^0I_{n+1}+\sum_{k=1}^n2a_{ij}^kE_{k0}+
\sum_{k=1}^n2\overline{a^k_{ji}}E_{0k})+(J_{ij})\in M_p(\cl{E}_n)^+ .\]

Now we let $J_{ij}=b_{ij}(E_{00}-\sum_{k=1}^nE_{kk})$ and $M$
corresponds to
\[ (a_{ij}^0)\otimes I_{n+1}+\sum_{k=1}^n2(a_{ij}^k)\otimes E_{k0}+\sum_{k=1}^n2(\overline{a^k_{ji}})\otimes E_{0k}+(b_{ij})\otimes (E_{00}-\sum_{k=1}^nE_{kk})
\in M_p(\cl{E}_n)^+ .\]

Finally, using the isomorphism $M_p(\cl{E}_n)\cong \cl{E}_n(M_p)$, we know that $M$ 
corresponds to the following positive block matrix in $M_{n+1}(M_p)$,
\[ \begin{pmatrix}
     A_0&2A_1^*&\cdots&2A_n^* \\
     2A_1&A_0&& \\
     \vdots&&\ddots&\\
     2A_n&&&A_0
    \end{pmatrix}+
    \begin{pmatrix}
     B&&&\\
     &-B&&\\
     &&\ddots&\\
     &&&-B
     \end{pmatrix}, \]
     where $B=(b_{ij})$.
 Since isomorphism is used in each step, we know that our conclusion is in fact an ``if and only if'' statement.
\end{proof}

\section*{Acknowledgments}
The author would like to thank Vern I.Paulsen for his inspiring advice and careful reading of the manuscript which lead to many improvements of this paper, and Mark
Tomforde for his helpful suggestions. Also, the author is grateful to the referees for their valuable comments.

  \bibliographystyle{amsplain}
\bibliography{operatorsystem}

 \end{document}